\documentclass[12pt]{amsart}

\usepackage[colorlinks, pagebackref, linkcolor=red, citecolor=blue]{hyperref}
\usepackage{amsmath,amssymb}
\usepackage{color}
\usepackage{mathrsfs}
\usepackage{graphicx}

\usepackage{calc}
\makeatletter
 
\newlength{\temp@wc@width}

\newlength{\temp@wc@height}

\newcommand{\widecheck}[1]{%
\setlength{\temp@wc@width}{\widthof{$#1$}}%
\setlength{\temp@wc@height}{\heightof{$#1$}}%
#1\hspace{-\temp@wc@width}%
\raisebox{\temp@wc@height+2pt}[\heightof{$\widehat{#1}$}]%
{\rotatebox[origin=c]{180}{\vbox to 0pt{\hbox{$\widehat{\hphantom{#1}}$}}}}%
}

\makeatother
 
\numberwithin{equation}{section}
\theoremstyle{plain}
\newtheorem{theorem}{Theorem}[section]
\newtheorem{lemma}[theorem]{Lemma}

\newtheorem{proposition}[theorem]{Proposition}

\theoremstyle{definition}
\newtheorem{definition} [theorem] {Definition}

\newtheorem{example}[theorem]{Example}

\def\bs{\boldsymbol}
 
\def\T{ \mathbb T}
\def\R{ \mathbb R}
\def\H{H^\infty}

\def\D{{ \mathbb D}}

\def\C{{ \mathbb C}}

\def\N{{ \mathbb N}}
\def\B{{ \mathbb B}}

\def\e{\varepsilon}

\def\bs{\boldsymbol}

\def\union{\cup}

\def\inter{\cap}
\def\Inter{\bigcap }
\def\ov{\overline}

\def\ss{\subseteq}
\def\emp{\emptyset}

\def\buildrel#1_#2^#3{\mathrel{\mathop{\kern 0pt#1}\limits_{#2}^{#3}}}
 
\def\BP{Blaschke product}

\def\IBP{interpolating Blaschke product}

\begin{document}



\title[Noncoherent uniform algebras] {Noncoherent uniform algebras in $\C^n$}
 
 \author{Raymond Mortini}
  \address{
   \small Universit\'{e} de Lorraine\\
\small D\'{e}partement de Math\'{e}matiques et  
Institut \'Elie Cartan de Lorraine,  UMR 7502\\
\small Ile du Saulcy\\
 \small F-57045 Metz, France} 
 \email{raymond.mortini@univ-lorraine.fr}

\begin{abstract}
Let $\mathbf D=\ov{\D}$ be the closed unit disk in $\C$ and 
$\mathbf B_n=\ov\B_n$ the closed unit ball in $\C^n$.
For a compact subset $K$ in $\C^n$ with nonempty interior, let $A(K)$ be the  uniform 
algebra of all complex-valued continuous functions on $K$ that are holomorphic 
 in the interior of $K$. We give short and  non-technical proofs of the known facts that 
 $A(\overline{\D}^n)$ and $A(\mathbf B_n)$  are noncoherent rings.   Using, additionally,
 Earl's interpolation theorem in the unit disk and the existence of peak-functions, we  
 also establish with the same method  the new result that  $A(K)$ is not coherent. 
 As special cases we obtain  Hickel's theorems  on the noncoherence  of $A(\ov \Omega)$, where 
 $\Omega$ runs through a certain  class of   pseudoconvex domains in $\C^n$, results
  that were obtained with  deep and complicated methods. Finally, using a refinement of 
 the interpolation theorem we show that no uniformly closed subalgebra $A$ of $C(K)$
 with $P(K)\ss A\ss C(K)$ is coherent provided the polynomial convex hull of
 $K$ has no isolated points. 
\end{abstract}

\maketitle

 \centerline {\small\the\day.\the \month.\the\year} \medskip

\section{Introduction}

In this paper we are interested  in a certain algebraic property of  some standard Banach algebras 
 of holomorphic functions of several complex variables.  By introducing new methods
 we are able to solve a fourty year old problem first considered by McVoy and Rubel 
 in the realm of  uniform algebras  appearing in approximation theory and complex analysis
 of several variables.
 
 Let us start by 
recalling the notion of a coherent ring. 

\begin{definition}
A  commutative unital ring $\mathcal A$  is said to be {\em coherent}  if
 the intersection of any two finitely generated 
ideals in $\mathcal A$ is finitely generated.
\end{definition}

We refer the reader to the article \cite{Gla} for the relevance of the property 
of coherence in commutative algebra.  

\begin{definition}
  Let $\D:=\{z\in \C: |z|< 1\}$ be the open unit disk in $\C$ and 
  $\B_n=\{\bs z=(z_1,\dots,z_n)\in \C^n:\sum_{j=1}^n |z_j|^2<1\}$  the open
  unit ball in $\C^n$.  Their Euclidean closures are denoted by $\mathbf D$ and 
  $\mathbf B_n$, respectively.
  
  For a bounded  open set $\Omega$ in $\C^n$, let 
   $H^\infty(\Omega)$ be the Banach algebra 
  of all bounded and holomorphic functions $f:\Omega\rightarrow \C$, 
  with pointwise addition and multiplication, and the supremum norm:
  $$
  \|f\|_\infty:=\sup_{z\in \Omega}|f(z)|, \quad f\in \H(\Omega).
  $$
  
  For a compact set $K\subset\C^n$, let  $A(K)$ be the  uniform 
algebra of all complex-valued continuous functions on $K$ that are holomorphic 
 in the interior  $K^\circ$ of $K$.  If $K=\ov\Omega$, then we view $A(K)$ 
 as a subalgebra of $\H(\Omega)$.
 
 If $K={\mathbf D}^n$, then $A(K)$ is called the {\em polydisk algebra}; if $K=\mathbf B_n$,
 then $A(K)$ is the {\em ball algebra}.
 \end{definition}

In the context of function algebras of holomorphic functions 
in the unit disk $\D$ in $\C$, we mention \cite{McVRub}, 
where it was shown that the Hardy algebra $H^\infty(\D)$ 
is coherent, while the disk algebra $A(\ov{\D})$ isn't. 
For $n\geq 3$, Amar \cite{Ama} showed that the Hardy algebras 
$H^\infty (\D^n)$, $\H(\B_n)$,  the polydisk algebra $A({\mathbf D}^n)$
and the ball algebra $A(\mathbf B_n)$  are not coherent. 

 The missing $n=2$ case for the bidisk algebra $A({\mathbf D}^2)$ (respectively the ball algebra $A(\mathbf B_2)$)
 follows as a special case of a general result due to Hickel \cite{Hic} on 
 the noncoherence of the algebra $A(\overline{\Omega})$ of continuous functions on $\overline{\Omega}$ that are 
 holomorphic in $\Omega$,   where $\Omega\subset \C^n$ ($n\geq 2$) is a bounded strictly
  pseudoconvex domain with a 
 $C^\infty$ boundary. But  
  the proof in  \cite{Hic} is technical. To illustrate our subsequent methods, 
  we first give a short, elegant proof of the 
  noncoherence of $A({\mathbf D}^n)$  and $A(\mathbf B_n)$. Let me  mention that
an entirely elementary proof, developed {\bf after} this manuscript
had been written in 2013, has been published in \cite{mo-sa}.
  
   Using techniques from the theory of Banach algebras  which are
 based on peak-functions,  bounded approximate 
 identities and Cohen's factorization theorem (compare with \cite{Morvon}), and, 
  additionally, function theoretic tools, as  Earl's interpolation theorem for  $\H(\D)$ in the unit disk  
   (a refinement of Carleson's interpolation theorem) \cite[p. 309] {ga},
 we succeed to  show the noncoherence of $A(K)$ for every compact set
 $K$ in $\C^n$.  
  
 Finally, by replacing Earl's theorem with a result on asymptotic interpolation,  we
 can handle  for   compact sets $K\ss\C^n$ without isolated points
the case of any uniformly closed algebra $A$ with $P(K)\ss A\ss C(K)$,
 where $P(K)$ is the smallest  closed subalgebra of $C(K)$ containing the polynomials.

To conclude, let me point out  that the coherence of rings of stable transfer functions of multidimensional systems, such as $A(\mathbf B_n)$ or 
$A({\mathbf D}^n)$, plays a role in the stabilization problem in Control Theory via the factorization approach; see \cite{Qua}.

\section{Preliminaries}
 
In this section we collect some technical results which we will use in the proof of our main results. 
 
\begin{lemma}
\label{lemma_1}
 Let $\mathcal A$ be a commutative unital ring and  
 $M$  an ideal in $\mathcal A$ such that $M\not=\mathcal A$.  Suppose that
$I$ is a finitely generated ideal of $\mathcal A$ which satisfies  $I=IM$. Then there exists $m\in M$ such that $(1+m)I=0$. 
If $\mathcal A$  has no zero divisors, then $I=0$.  
\end{lemma}
\begin{proof} This follows from Nakayama's lemma \cite[Theorem~76]{Kap}.
\end{proof}

 \begin{lemma}\label{zerod}
 Let $I$ be a non-finitely  generated ideal in a commutative unital ring $\mathcal A$.
  Suppose that  $a\in \mathcal A$
 is not a zero-divisor. Then $a I$ is not finitely generated either.
  \end{lemma}
 \begin{proof}

Suppose, on the contrary,  that $a I=(G_1,\dots,G_m)$, for some elements $G_1,\dots, G_m$ in $\mathcal A$.
 Then 
there exist elements $F_1,\dots, F_m \in I$ such that $G_j=a F_j$, $j=1,\dots, m$. 
We claim that $I=(F_1,\dots, F_m)$. 
Indeed, trivially $(F_1,\dots, F_m)\ss I$. Also, 
 for any $f\in I$, $af\in aI=(G_1,\dots ,G_m)$ gives the
  existence  of $\alpha_1,\dots,\alpha_m\in \mathcal A$ such that 
  $$
 af=\alpha_1 G_1+\cdots+\alpha_m G_m = \alpha_1 aF_1+ \cdots +\alpha_m aF_m.
 $$ 
 Since $a$ is not a zero-divisor, 
 it follows that  
 $$
 f=\alpha_1 F_1+\dots +\alpha_m F_m\in (F_1,\dots, F_m).
 $$
 This shows that the reverse inclusion $I\ss (F_1,\dots, F_m)$ is true, too. 
  But this means that $I$, which coincides with $(F_1,\dots, F_m)$, is finitely generated, a contradiction. 
\end{proof}

 Here is an example that shows that  the condition on $a$ being
a non-zero-divisor is necessary:
\begin{example}
Let $D_1$ and $D_2$ be two disjoint copies of the unit disk, say $D_1=\{|z-0.5|<0.5\}$
and $D_2=\{|z+0.5|<0.5\}$,
 and let $\mathcal A$ be 
the algebra of bounded analytic functions on $D_1\union D_2$. 
Let $S(z)=\exp(- (1+z)/(1-z))$ be the atomic inner function.
Consider the associated elements  $f_n$ of $\mathcal A$ given by
$$f_n(z)= \begin{cases} S^{1/n}(z) & \text{if $z\in D_1$}\\ 
                                       S(z) &\text{if $z\in D_2$}.
                                       \end{cases}
                                       $$
and let the function $a\in \mathcal A$ be defined as
$$a(z)=\begin{cases}  0 & \text{if $z\in D_1$}\\
                                     1 & \text{if $z\in D_2$}.
                                     \end{cases}
                                     $$
Then  the ideal $I=(f_1,f_2,\dots,)$ generated by the functions $f_n$ in $\mathcal A$ is not
finitely generated, although the ideal $ a I$  is finitely generated.
\end{example}

 \begin{definition}
Let $X$ be a metrizable space. 
\begin{enumerate}
\item [(1)]  $C_b(X,\C)$ denotes 
 the space of  bounded, complex-valued continuous functions on $X$.
 \item [(2)] A {\em function algebra} $A$ on $X$ is 
  a uniformly closed, point separating subalgebra of $C_b(X,\C)$, 
   containing the constants. 
   \item [(3)] A  point $x_0\in X$ is called a peak-point for $A$, if there is a
   function $p\in A$ (called  a {\em peak-function})
    with $p(x_0)=1$ 
    and
   \begin{equation}\label{sup}
\sup_{x\in X\setminus U}|p(x)|<1
\end{equation}
    for every open neighborhood  $U$ of $x$.
   \end{enumerate}
 \end{definition}
 Note that in case $X$ is compact, condition (\ref{sup}) is equivalent to
 $$\mbox{$|p(x)|<1$ for all $x\in X, x\not=x_0.$}$$
 
\begin{definition}
 Let $\mathcal A$ be a commutative Banach algebra (without an identity element), and $M$  a closed ideal of $\mathcal A$.
 Then a bounded sequence $(e_n)_{n\in \N}$ in $M$ 
 is called a (strong) {\em approximate identity for $M$} if 
 $$
 \lim_{n\rightarrow \infty} \|e_n f-f\|=0
 $$
 for all $f\in M$. 
\end{definition}
For compact spaces,  the following Proposition is 
in  \cite[p. 74, Corollary 1.6.4]{Bro}. 

 \begin{proposition}\label{peakba}
 Let $X$ be  a metric space and 
  $x_0\in X$   a peak-point for the function algebra $A$ on $X$.   If  $p$  is an associated 
  peak function, then the sequence $(e_n)$ defined by
 $$e_n=1-p^n$$
 is a bounded approximate identity for the maximal ideal
 $$M(x_0)=\{f\in A: f(x_0)=0\}.$$
\end{proposition}

\begin{proof}
For the reader's convenience
here is the outline:\\
 In fact, for $f\in A$,
$$
|e_k f-f|=|p|^k |f|.
$$
Let $\epsilon>0$. As $f(x_0)=0$, 
there is an open neighbourhood $U$ of $x_0$ such that $|f|<\epsilon$ on $U$. 
By assumption, 
$$m:=\sup_{X\setminus U}|p|<1.$$
 Now choose $k_0\in \N$ large enough so that for $k>k_0$, $m^k \|f\|_\infty <\epsilon$. Thus for $k>k_0$, 
 $$
 |e_k f-f|=|p^k f|\leq \left\{\begin{array}{ll} m^k \|f\|_\infty &\textrm{on } X\setminus U\\
                               1^k \cdot \epsilon&\textrm{on } U\\
                              \end{array}\right\}<\epsilon. 
                              $$
                              Hence $\|e_kf-f\|_\infty\leq \e$ for $k>k_0$.
 \end{proof}

Our central Banach-algebraic tool will be  Cohen's Factorization Theorem;  
see \cite[p.74, Theorem~1.6.5]{Bro}. 

\begin{proposition}\label{brco}
 Let $\mathcal A$ be a commutative unital  real or complex Banach algebra, $I$  a closed ideal of $\mathcal A$, and suppose that $I$ 
 has an approximate identity. Then every $f\in I$  can be decomposed in
 a product $f=gh$ of two functions $g,h\in I$.
 
 \end{proposition}
 
 The main function-theoretic tool for the construction of  our ideals in general
 uniform algebras in $\C^n$ will be the following result on asymptotic
 interpolation  given in \cite[p. 515]{mo}, with predecessors in \cite{gomo} and \cite{dn}.
 Recall that $\rho(z,w)=|(z-w)/(1-\ov z w)|$ is the pseudohyperbolic distance between
 $z$ and $w$ in $\D$.

 \begin{theorem}\label{asy}
 Let $(a_n)$ be  a thin sequence in $\D$; that is a sequence such that the associated
 Blaschke  product $b$ satisfies
 $$\lim_n (1-|a_n|^2)|b'(a_n)|=1.$$
 Then for any sequence $(w_n)\in \ell^\infty$  with $\sup_n |w_n|\leq 1$ there exists
 a \BP\ $B$  and  a sequence of positive numbers  $\tau_n\to 1$ such that  
 for any $0\leq \tau_n'\leq \tau_n$ with $\tau_n'\to 1$,
 the zeros of $B$ can be chosen to be contained in the  union of the pseudohyperbolic disks
 $\{z\in \D: \rho(z,a_n)\leq \tau_n'\}$ and such that 
 $$|B(a_n)-w_n|\to 0.$$  
 If the interpolating nodes $(a_n)$ cluster only at the point 1, then the zeros
 of $B$ can be chosen so that they cluster also only at 1.
 \end{theorem}
 \begin{proof}
 It  remains to verify  the assertion on the zeros of $B$ whenever $(a_n)$ clusters only at $1$.
 Since  the pseudoyperbolic disk $D_\rho(a,r)$  coincides with the Euclidean disk
 $D(C, R)$ where
 $$C= \frac{1-r^2}{1-r^2 |a|^2} a$$ and
 $$R= \frac{1-|a|^2}{1-r^2|a|^2} r$$
 (see \cite{ga})
 it suffices to choose $\tau_n':= \min\{\tau_n, r_n\}$, where 
 $$r_n=\sqrt{ \frac{1-\sqrt{1-|a_n|^2}}{|a_n|^2}}$$
 and to verify that in that case
 $R_n\to 0$ and $C_n\to 1$.  
  \end{proof}

\section{ A sufficient criteria for noncoherence}

The following  concept of multipliers  is new and is the key for  our short proofs 
of the noncoherence results. 

\begin{definition} 
Let $A$ be a function algebra on a  metrizable space $X$ and $x_0\in X$ a non-isolated point 
 \footnote{ This means that there is a sequence of distinct points in $X$ converging to $x_0$.}.
A function 
$$S\in  C_b\bigl(X\setminus\{x_0\},\C\bigr)$$ 
 is called a {\em multiplier} for the maximal ideal
$$M(x_0)=\{f\in A: f(x_0)=0\},$$
if the ideal
$$L:=L_S:=\{f\in A: Sf\in A\}$$
coincides  with $M(x_0)$ and
if  there exists $p\in M(x_0)$ such that  $pS$ is not a zero-divisor \footnote{ The notation $Sf\in A$
is to be interpreted in the usual way that $Sf:X\setminus\{x_0\}\to\C$ has a continuous extension
$F$ to $X$ with $F\in A$}.
\end{definition}

As a canonical example we mention the atomic inner function 
$$S(z)=\exp\left(- \frac{1+z}{1-z}\right),$$
which is a multiplier for the maximal ideal $M(1)$ of the disk algebra $A(\mathbf D)$.

\begin{theorem}\label{coh}
Let $A$ be a function algebra on a  metrizable space $X$.
Suppose that $x_0\in X$ is a non-isolated  peak-point for $A$ 
and that the function
$S\in C_b( X\setminus\{x_0\},\C)$  is a multiplier for the maximal ideal
$M(x_0)$. Then $A$ is not coherent.
\end{theorem}

\begin{proof}

We shall unveil two principal ideals whose intersection is not finitely generated. 
By assumption, $S$ is  a multiplier for $M(x_0)$. In particular,  there is
a function $p\in M(x_0)$ so that $pS\in A$ is not a zero-divisor. This implies that $p$
is not a zero-divisor, either.
 Let  \begin{eqnarray*}
 I&:=&(p),\\
 J&:=&(pS),\\
 K&:=&\{pSf: f\in A \textrm{ and } Sf\in A\},\textrm{ and}\\
 L&:=&\{f\in A: Sf \in A\}.
 \end{eqnarray*}
 
 We claim that $K=I\cap J$. Trivially $K\ss I\cap J$. On the other hand, if 
 $g\in I\cap J$, then there exist $f,h\in A$ such that 
 $g=ph=pSf$, and so $Sf=h\in A$. In other words, $g\in K$. Thus also $I\cap J\ss K$. 
 
 It remains to show that $K$ is not finitely generated. Note that by definition,
 $K=pSL$. Moreover, since $S$ is a multiplier for $M$, we have
 $M=L$.   
 
 Let $f\in L$. Since $M$ has an approximate identity (by Proposition \ref{peakba}), we may apply 
 Cohen's factorization Theorem  (Proposition \ref{brco})  to conclude that there 
  exists $g, h\in M$ such that
 $$f=hg.$$
Consequently, $L=LM$.    Assuming that $L$ is finitely generated, there exists, 
by  Nakayama's Lemma \ref{lemma_1},
  $m\in M$ such that $(1+m)L=0$.  Note that $L=M$. Since $A$ is point separating,
 there exists for every $x_1\in X\setminus \{x_0\}$ a function 
 $f\in M=L$ such that $f(x_1)\not=0$.  Hence
 $(1+m(x_1)) f(x_1)=0$ implies that $m(x_1)=-1$.
 Since, by assumption,  $x_0$ is not an isolated point in $X$, the continuity of $m$
 on $X$ implies that $m(x_0)=-1$; a contradiction to the fact that $m\in M(x_0)$.
 Thus we conclude that $L$ cannot be finitely generated. 
 
Because $S$ is a multiplier,  $pS\in M(x_0)$. Moreover,    $pS$ is  not a zero-divisor. 
 Hence, by Lemma \ref{zerod},  $K=pSL$
 is not finitely generated either. 
\end{proof}

In the next sections we apply Theorem \ref{coh} to concrete function algebras of several
complex variables.

\section{The noncoherence of the ball and polydisk algebra}
In view of Theorem \ref{coh}, 
to prove the noncoherence,   it
 suffices to unveil a peak-function  and a multiplier for some distinguished maximal ideal.
\begin{theorem}
The ball algebra $A(\mathbf B_n)$ is not coherent for any $n=1,2,\dots$.
\end{theorem}
\begin{proof}
$$P(z_1,\dots,z_n)=  \frac{1+z_1}{2}$$
is a peak-function at $(1,0,\dots,0)$ for $A(\mathbf B_n)$ (note that if $|1+z_1|=2$, then $z_1=1$
and the remaining coordinates $z_2,\dots,z_n$ are automatically zero because
 $(z_1,\dots,z_n)\in \mathbf B_n$),  and
 $$
 S(z_1,\dots, z_n)=\exp \left(-\frac{1+z_1}{1-z_1}\right)
 $$
 is a multiplier for $M(1,0,\dots,0)$.
\end{proof}

\begin{theorem}
The polydisk algebra $A(\mathbf D^n)$ is not coherent for any $n=1,2,\dots$.
\end{theorem}
\begin{proof}
$$P_1(z_1,\dots,z_n)= \left( \frac{1+z_1}{2}\right)\cdots \left( \frac{1+z_n}{2}\right)$$
is  a peak-function at  $a=(1,\dots,1)$ for $A(\mathbf D^n)$ and
$$S(z_1,\dots,z_n)= \exp \left(-\frac{1+z_1}{1-z_1}\right)\cdots \exp \left(-\frac{1+z_n}{1-z_n}\right)$$
is a multiplier for $M(1,\dots,1)$.
\end{proof}

Thus we have obtained a short proof of this result by Amar and Hickel 
\cite{Ama, Hic}.

\section{The noncoherence of $P(K)\ss A\ss C(K)$}

For  a compact set $K\subset \C^n$, let
$C(K)$ denote the uniform algebra of complex-valued continuous functions on $K$,
 $A(K)$  the uniform algebra  of all functions continuous
on $K$ and holomorphic in $K^\circ$ and let $P(K)$ be the subalgebra
of those functions in $A(K)$ that can be uniformly approximated on $K$
by holomorphic polynomials. 

 Let us recall the following well-known result:
 
 \begin{theorem}\label{thealgebra}
Let $K\ss\C^n$ be a compact set. 
 Then the following assertions hold:
 
\begin{enumerate}
\item [(1)]  Endowed with the usual pointwise operations \footnote { addition $+$, 
multiplication $\cdot$
and multiplication $\buildrel\bullet_{s}^{}$ by complex scalars}
 and the supremum norm 
$$\|f\|_\infty=\sup \{|f(z)|: z\in K\}$$
$A(K)=A(+,\cdot, \buildrel\bullet_{s}^{}, \|\cdot\|_\infty)$ and 
$P(K)=A(+,\cdot, \buildrel\bullet_{s}^{}, \|\cdot\|_\infty)$
are uniformly closed point separating  subalgebras of  $C(K)$. 
 
\item [(2)]  Let $A$ be $A(K)$ or $P(K)$. Standard   maximal ideals  in $A$
 are given by 
$$M(z_0):=\{f\in A: f(z_0)=0\}$$
for a  uniquely determined $z_0\in K$.\footnote{ Note that, in general, there are many more
maximal ideals than those given by point-evaluation at points in $K$; even in the case  
 where $K=\ov\Omega$,
$\Omega$ a bounded pseudoconvex domain in $\C^n$, $n\geq 2$,
every function $f\in \H(\Omega)$ (a fortiori $f\in A(\ov \Omega)$)  may  
have a bounded holomorphic  extension to
a strictly larger domain $\Omega'$ (see \cite{JP}). Hence, in that case, the spectrum
of $A(\ov \Omega)$ is strictly larger than $\ov\Omega$ itself.}

\item[(3)]  The spectrum (or maximal ideal space) of $P(K)$ coincides with the
polynomial convex hull $\widehat K$ of $K$.

\item [(4)]   The Shilov-boundary,  $\partial A$, of $A$  is a non-void closed subset of  
$\partial K$. 
\item[(5)]   The set $\Pi(A)$ of peak-points for $A$  is a non-void dense subset of $\partial A$.

\item [(6)] For each $z_0\in \Pi(A)$,  the associated maximal ideal $M(z_0)$ has a bounded
approximate identity.  
\end{enumerate}
\end{theorem}
\begin{proof}
(1) is elementary; (2)-(5) are standard facts in the theory of uniform algebras
(see for instance  \cite{Bro} and  \cite{gam}); note that the Shilov-boundary is the closure
of the set of weak-peak points and that  for function algebras on  metrizable spaces 
every weak-peak point actually is a peak-point (\cite[p. 96]{Bro}).
(3) is in \cite[p. 67]{gam}.
(6)  follows from Proposition \ref{peakba}.
\end{proof}

We note that if $x_0\in\partial K$ is a peak-point for $P(K)$, then it is a peak-point
for any uniformly closed algebra $A$ with $P(K)\ss A\ss C(K)$.

\begin{lemma} \label{pp}
Let $K\ss\C^n$ be compact with $K^\circ\not=\emp$. 
 If $z_0\in \partial( K^\circ)$
is a peak-point for $P(\ov{K^\circ})$, then $z_0$ is a peak-point for $A(K)$.
\end{lemma}
\begin{proof}
Let $f\in P(\ov{K^\circ})$ peak at $z_0$.  Then $f(\ov{K^\circ})\ss \D\union\{1\}\ss\ov\D$.
Let $F:\C^n\to \C$ be a  continuous extension of $f$ to $\C^n$.  Since $\ov \D$ is  a retract for $\C$,
there is  a retraction map $r$ of $\C$ onto $\ov \D$ with $r(z)=z$ for $z\in \ov \D$.
Hence the function $r\circ F$ is an extension of $f$ with target space $\ov \D$.

By Urysohn's Lemma in metric spaces,  there is  a continuous function $u:\C^n\to\;[0,1]$  such
that
$$\{z\in \C^n: u(z)=1\}=\ov{K^\circ}.$$
Now consider $\phi(z)= (1+z)/2$ that maps $\ov \D$ onto $|z-1/2|\leq 1/2$. We claim that
$$g:=\phi\circ \bigl(u\cdot (r\circ F)\bigr): K\to \D\union\{1\}$$
 is a peak  function at $z_0$ that belongs to $A(K)$.

 To see this, we note that  $u(z)\cdot (r(F(z))\in \ov \D$ for every $z\in K$. Moreover,
 for $z\in K^\circ$,  $F(z)=f(z)\in\ov \D$; hence $r(F(z))=f(z)$ and so $u(z) r(F(z))=f(z)$.
 Since $\phi$  and $f$ are  holomorphic, we deduce that $g$ is holomorphic in $K^\circ$.
 Thus $g\in A(K)$. Now if for some $z_1\in K$, $g(z_1)=1$,  then necessarily 
 ${u(z_1)\,r(F(z_1))=1}$.
 Now $|r(F(z_1))|\leq 1$; hence $|u(z_1)|=u(z_1)=1$. We conclude that $z_1\in \ov{K^\circ}$.
 Therefore, as was shown previously, $u(z_1) r(F(z_1))=f(z_1)=1$.  Since  $f\in P(\ov{K^\circ})$
 peaks at $z_0$, we finally obtain that $z_1=z_0$.
   \end{proof}

\begin{definition}
Let $\Omega\subset \C^n$ be a bounded  open set.  For $a\in \partial \Omega$ and a function 
$f\in \H(\Omega)$, let $\textrm{Cl}(f, a)$ denote the {\em cluster set of $f$ at $a$}; that is
$\textrm{Cl}(f,a)$ is the set of all points $w\in \C$ such there exists a sequence $(z_n)$  in $\Omega$
such that $(f(z_n))$ converges to $w$.
\end{definition}

It is obvious  that $\textrm{Cl}(f,a)$ is a compact, nonvoid subset of $\C$.  In fact
$$
\textrm{Cl}(f,a)=\Inter_{0<r\leq 1} \ov {f(\Omega\inter B(a, r))}.
$$
In the case of the polydisk or unit ball, $\textrm{Cl}(f,a)$ is connected.

The proof of the following fundamental Lemma was motivated by  parts of the 
proof of  \cite[Theorem 3.1]{moru} 
concerning the  pseudo-B\'ezout property for $P(K)$ and its siblings, where $K\subset \C$
 is compact.
It gives us the possibility to construct multipliers for maximal ideals.

 \begin{lemma}\label{admissible}
 For a compact  set $K\ss \C^n$, let $A$ be  a uniformly
 closed algebra with $P(K)\ss A\ss C(K)$. 
Let $x_0\in\partial K$ be a non-isolated  peak-point for $P(K)$ and
 $p\in P(K)$ an associated peak-function.
  Then there exists  a function $S\in C_b(K\setminus \{x_0\})$ such that
  $$ 
  0\in \textrm{\em Cl}(S, {x_0}) \text{\;\;but \;$\textrm{\em Cl}(S, {x_0})\not=\{0\}$},
  $$
  and 
  $$(1-p) S\in A.$$
  Moreover, $S$ is a multiplier for  the maximal ideal
  $$M(x_0)=\{f\in A: f(x_0)=0\}.$$
  
\end{lemma}

\begin{proof}
{\bf Case 1} We first deal with the case, where $K$ is the closure of  a domain $D$ in $\C^n$
\footnote{ This  is only for the purpose of simplicity, because the tools applied in this case
 are more elementary than in the general case.}.

Let $(z_n)\in K$  be  a sequence of distinct points in $K$
converging to ${x_0}$. 
Then $p(z_n)\to 1$ and $p(z_n)\in \D$.   
By passing to a subsequence, if necessary, we may assume that
$(p(z_n))$ is a (thin) interpolating sequence for $\H(\D)$. 
Using Earl's interpolation
theorem \cite[p. 309]{ga},  there is  an \IBP\ $B$ satisfying  
\begin{equation}\label{earl}
\mbox{$B(p(z_{2n}))=0$ and $B(p(z_{2n+1}))=\delta$}
\end{equation}
 for all $n$ and some constant $\delta>0$ and such that
the zeros of $B$ cluster only at $1$.   Hence $B\circ p$ is discontinuous at ${x_0}$.


Now let $S:= B\circ p$.  
Since $|p|<1$ everywhere on $K\setminus \{x_0\}$, it follows from the fact that
$B$ is continuous on $\ov\D\setminus \{1\}$ that $S=B\circ p$ is continuous on 
$K\setminus \{x_0\}$.
Moreover, since $x_0$ is not an isolated point, $0\in {\rm Cl}(S, {x_0})$ and 
$\delta \in {\rm Cl}(S, {x_0})$. 


It remains to show that  $(1-p)S\in A$ and that $S$ is the multiplier  we are looking for.
Let us point out that for any 
$q\in C(\ov \Omega)$ with $q({x_0})=0$, the function  $qS=q\cdot (B\circ p)$ is 
continuous at ${x_0}$.
We claim that if $q\in A$ and $q(x_0)=0$, 
 then $q(B\circ p)\in A$. 
 
 To this end, consider
the partial products $B_n:=\prod_{j=1}^n L_j$ of the \BP\ $B$. Then   
$B_n$ converges locally uniformly (in $\D$) to $B$. Since $B_n$ is analytic in a neighborhood
of the $P(K)$-spectrum $\sigma(p)$ of $p$, where $\sigma(p)\ss\ov\D$, 
we see that $B_n\circ p\in P(K)\ss A$.  Now
$ q (B_n\circ p)$ converges uniformly in $K$ to $q(B\circ p)$. Hence $q(B\circ p)\in A$.
In particular, $(1-p)(B\circ p)\in A$.

Thus we have shown that 

$$ M(x_0)\ss I_S:=\{f\in A: Sf\in A\}.$$
To show the reverse inclusion, let $f\in I_S$.
 Then the continuity of $f$ and the discontinuity of $S$ at ${x_0}$
 imply that $f({x_0})=0$.  Hence $f\in M(x_0)$ and so $I_S\ss M(x_0)$. 
 Consequently
 $$I_S=\{f\in A: Sf\in A\}= M(x_0).$$
 
 To show that $(1-p)S$ is not a zero-divisor in $A$, we  have to use the special structure of $K$,
 namely that $K=\ov D$ for a domain $D$ in $\C^n$. 
 Note that  for general $K$, $S=B\circ p$  may vanish identically on whole components
 of $K^\circ$ (for example if $B$ has a zero at $p(a)$ and $p\equiv p(a)$ on such a component).
 Now $(1-p)S$ is analytic on $D$; since its zeros are isolated, we deduce that $(1-p)S q\equiv 0$
 implies $q\equiv 0$ on $D$ for every $q\in A$. 
 
  Putting it all together, we have shown that $S$ is  a multiplier for $M(x_0)$.\\
  
  {\bf Case 2} Now let $K\ss\C^n$ be an arbitrary compact set.  To avoid the phenomenon
  described in the last paragraph, we have to look for a multiplier  $S$ that has no zeros
  on $K\setminus \{x_0\}$.  It will have the form
  $$S=(1+B)\circ p= 1+ (B\circ p)$$
  for some Blaschke product $B$ whose zeros cluster only at $1$.  
  Note that $B\circ p$  does never take the value $-1$ on
  $K\setminus \{x_0\}$, since $B(\xi)=-1$ only for  $\xi\in\T\setminus \{1\}$ and
  the only unimodular value $p$ takes, is $1$.  
  
  Here is now the construction of $B$. According to the asymptotic interpolation
  theorem \ref{asy}, there is a Blaschke product $B$  whose zeros cluster only at $1$
  such that
  \begin{equation}\label{dnm}
  \mbox{$B(p(z_{2n}))\to -1$ and $B(p(z_{2n-1}))\to 1$}.
\end{equation}
Hence
$0\in {\rm Cl}(S, {x_0})$ and 
$2\in {\rm Cl}(S, {x_0})$.  The rest is now clear in view of the proof of Case 1,
 always having in mind that $x_0$ is not an isolated point in $K$.
\end{proof}

\begin{theorem}\label{main}\hfill

i) If $\Omega$ is a bounded domain in $\C^n$, then $A(\ov\Omega)$  is not coherent.
\noindent

ii) If $K\subset \C^n$ is compact with $K^\circ\not=\emp$, then $A(K)$ is not coherent.
\end{theorem}

\begin{proof}
i)  By Theorem \ref{thealgebra},  there exists a peak-point $z_0\in\partial\ov\Omega$
for $A(\ov\Omega)$ and  $M(z_0)$ has an approximate identity.  Of course
 $\ov\Omega$ is a compact set without isolated points.
Hence, by Lemma \ref{admissible},
there is a multiplier $S$ for $M(z_0)$. The noncoherence of   $A(\ov\Omega)$
 now follows from Theorem \ref{coh}.
 
 ii) Similar as i); just  use Lemma \ref{pp} to get the non-isolated peak-point $x_0$ for $A(K)$.
\end{proof}

If $K^\circ=\emp$, then $A(K)=C(K)$. In Section \ref{noncoho} we will give a characterization
of those  compacta in $\C^n$ for which $C(K)$ is coherent. Let us also note that  i) is not a
 special case of ii), because there are algebras of the form $A(\ov \Omega)$ that 
 do not belong to the class of algebras of type $A(K)$:  just take as $\Omega$  the unit disk
 deleted by a Cantor set (=compact and totally disconnected) of positive planar Lebesgue measure.

\begin{definition}
A compact set $K\ss\C^n$ is called {\em admissible} if its polynomial convex hull $\widehat K$
does not contain any isolated points.
\end{definition}
Our final theorem contains (more or less) all the preceding ones as a special case.

\begin{theorem}
Let $K\ss\C^n$ be an admissible compact set and let $A$ be a uniformly closed subalgebra of $C(K)$
with $P(K)\ss A\ss C(K)$. Then $A$ is not coherent.
\end{theorem}
\begin{proof}
Similar as the proof above; note that Theorem \ref{thealgebra} (5) yields the desired
peak-point for $P(K)$ and Lemma \ref{admissible} the associated multiplier.
\end{proof}

If we are considering algebras of a single complex variable, then we have the following 
refinement:
\begin{theorem}
Let $K\ss\C$ be an infinite  compact set and let $A$ be a uniformly closed subalgebra of $C(K)$
with $P(K)\ss A\ss C(K)$. Then $A$ is not coherent.
\end{theorem}
\begin{proof}
The infinity of $K$ implies that the polynomial convex hull $\widehat K$ of $K$ is an infinite
compact set, too. This in turn implies that its topological boundary $\partial\widehat K$
is an infinite compact set. Hence, there exists a non-isolated point $x_0\in \partial\widehat K\ss K$.
Now by Mergelyan's Theorem $P(\widehat K)=R(\widehat K)=A(\widehat K)$. 
Using the fact that $\C\setminus \widehat K$
is connected, Gonchar's peak-point criterium for $R(K)$ (see \cite[Corollary 4.4, p. 205]{gam}) 
shows that $x_0$ is a peak-point for $R(\widehat K)=P(\widehat K)$. The non-coherence now follows 
as in the preceding theorems.
\end{proof}

We guess that this result can be extended to the case of several variables.

 \section{Noncoherence of $C(K)$}\label{noncoho}
 
Let $K$ be a compact set in $\C^n$.  A general result in \cite{nev}
  tells us that for completely regular spaces $X$, $C(X,\R)$ is coherent if and only if
 $X$ is basically disconnected \footnote{ Recall that $X$ is said to be basically disconnected if
 the closure  of $\{x\in X: f(x)\not=0\}$ is open for every $f\in C(X,\R)$.}.
  This result can be used  to conclude that
 $C(K,\C)$ is coherent if and only if $K$ is finite.  
 Since our compacta $K$ are metrizable, we would like to present, for the reader's convenience,  
 the following independent easy proof.

 \begin{theorem}
 If $K\ss\C^n$ is compact, then $C(K)$ is coherent if and only if $K$ is finite.
\end{theorem}
\begin{proof}
Suppose that $K$ is not finite. Then there is ${x_0}\in K$ such that $\lim x_n={x_0}$
for some sequence $(x_n)$ of distinct points in $K$. 
Let  $E$ be a closed subset of $K$ not containing ${x_0}$.
Then 
$$p(x)=  \frac{d(x,E)}{d(x,E)+d(x,{x_0})}$$
is a peak-function for $x_0$.   By passing to a subsequence, if necessary, we may assume
that $p(x_n)\not=p(x_m)$ for $n\not=m$.
Choose a continuous zero-free function $B:[0,1[\to \;]0,1]$  such that 
$$\mbox{$B(p(x_{2n}))=1$ and $B(p(x_{2n-1}))=1/n\to 0$,}$$
and let 
$$S:= B\circ p.$$
Then $S\in C_b(K\setminus\{x_0\})$.
It is now straightforward to check that the continuous function
 $(1-p)S$ is not a zero-divisor and that $S$ is a multiplier
for $M(x_0)$
(note that the cluster set of $S$ at $x_0$ is not a singleton and contains $0$).
Then we apply Theorem \ref{coh}.

If, on the other hand, $X$ is finite, then $C(X)$ is a principal ideal ring. In fact,
 if $I\ss C(X)$ is an ideal, then we define a generator $g$ of $I$
 by  $g(x)=1$ if $x\notin Z(I)$ and $g(x)=0$ if $x\in Z(I)$.
 Hence $C(X)$ is trivially coherent. 
\end{proof}

\subsection*{Acknowledgements}

 I thank Amol Sasane for many discussions  on noncoherence and 
 Peter Pflug for some valuable comments concerning extensions
 of bounded holomorphic functions in several variables, and for providing reference \cite{JP}.

\end{document}